\documentclass[11pt]{amsart}
\usepackage{amssymb,mathrsfs,graphicx,enumerate}
\usepackage{amsmath,amsfonts,amssymb,amscd,amsthm,bbm}
\usepackage{graphicx,colortbl}
\usepackage{mathtools}

\title[Convergence of consensus-based global optimization method]{Convergence  of a first-order consensus-based global optimization algorithm}

\author[Ha]{Seung-Yeal Ha}
\address[Seung-Yeal Ha]{\newline Department of Mathematical Sciences and Research Institute of Mathematics, \newline
Seoul National University, Seoul, 08826 and \newline
Korea Institute for Advanced Study, Hoegiro 85, Seoul, 02455, Korea (Republic of)}
\email{syha@snu.ac.kr}

\author[Jin]{Shi Jin}
\address[Shi Jin]{\newline School of Mathematical Sciences, MOE-LSC, and Institute of Natural Sciences, \newline Shanghai Jiao Tong University, Shanghai 200240, China}
\email{shijin-m@sjtu.edu.cn}

\author[Kim]{Doheon Kim}
\address[Doheon Kim]{\newline School of Mathematics, Korea Institute for Advanced Study,\newline Hoegiro 85, Seoul 02455, Korea (Republic of)}
\email{doheonkim@kias.re.kr}

\newtheorem{theorem}{Theorem}[section]
\newtheorem{lemma}{Lemma}[section]

\newtheorem{remark}{Remark}[section]

\newcommand{\bbr}{\mathbb R}

\newcommand{\bbe}{\mathbb E}

\newcommand{\bbp} {\mathbb P}

\begin{document}

\date{\today}



\thanks{\textbf{Acknowledgment.} The work of S.-Y. Ha was supported by National Research Foundation of Korea(NRF-2017R1A2B2001864), and the work of S. Jin was supported by NSFC Grant Nos. 11871297 and  3157107.}

\begin{abstract}
Global optimization of a non-convex objective function often appears in large-scale machine-learning and artificial intelligence  applications. Recently, consensus-based optimization (in short CBO) methods have been introduced as one of the gradient-free optimization methods. In this paper, we provide a convergence analysis for the first-order CBO method in \cite{C-J-L-Z}. Prior to the current work, the convergence study was carried out for CBO methods on corresponding mean-field limit, a Fokker-Planck equation, which does not imply the convergence of the CBO method {\it per se}. Based on the consensus estimate directly on  the first-order CBO model, we provide a  convergence analysis of the first-order CBO method \cite{C-J-L-Z} without resorting to the corresponding mean-field model. Our convergence analysis consists of two steps. In the first step, we show that the CBO model exhibits a global consensus time asymptotically for any initial data, and in the second step, we provide a sufficient condition on system parameters--which is dimension independent-- and initial data which guarantee that the converged consensus state lies in a small neighborhood of the global minimum almost surely. 
\end{abstract}
\maketitle \centerline{\date}


\section{Introduction} \label{sec:1}
\setcounter{equation}{0}
Large-scale optimization problems often appear in machine learning and artificial intelligence (AI) applications, in which objective functions to be optimized are not necessarily convex nor regular enough, say ${\mathcal C}^1$ in general. Thus, one might not be able to use the standard stochastic gradient descent method. Alternatively, several gradient-free optimization methods based on collective dynamics are used in application domain, for example swarm intelligence methods \cite{Ke, Y-D} such as particle swarm optimization (in short PSO) \cite{E-K}, simulated annealing method \cite{K-G-V, L-A}, ant-colony algorithm \cite{Ya}, genetic algorithm \cite{Ho}  etc. A basic idea of these metaheuristic algorithm is to use collective behaviors of underlying sample points coupled with suitable stochastic components in the choice of system parameters.  Despite of its usefulness, rigorous convergence analysis of such swarm intelligence algorithms is often missing. 

In this paper we provide a convergence analysis for the CBO method proposed in \cite{C-J-L-Z}. To be more specific, let $X_t^k = (x_t^{k,1}, \cdots, x_t^{k,d}) \in \bbr^d$ be the coordinate process of the $k$-th sample point at time $t$. Suppose that one looks for a global minimum $X^* \in \bbr^d$ for a given objective function $L$:
\[  \displaystyle  X^* \in \mbox{argmin}_{X \in \bbr^d} L(X), \]
where the objective function $L$ may be neither convex nor smooth enough. In this situation, gradient-based optimization methods such as the stochastic gradient descent method \cite{Be} can not be used as it is. In a recent work \cite{C-J-L-Z}, the authors proposed the following variant of the CBO algorithm introduced in \cite{C-C-T-T, P-T-T-M, T-P-B-S}: 
\begin{equation} \label{A-1}
\begin{cases}
\displaystyle dX^i_t = -\lambda (X^i_t - {\bar X}_t^*)dt  + \sigma  \sum_{l=1}^{d} (x^{i,l}_t - {\bar x}_t^{*,l}) dW_t^l e_l, \quad t > 0,~~ i = 1, \cdots, N, \\
\displaystyle {\bar X}_t^* =(x_t^{*,1}, \cdots, x_t^{*,d})  := \frac{\sum_{l=1}^{N} X^l_t e^{-\beta L(X^l_t)}}{\sum_{l=1}^{N} e^{-\beta L(X^l_t)}},
\end{cases}
\end{equation}
where $\lambda$ and $\sigma$ denote the drift rate and noise intensity, respectively, and  $\beta > 0$ is a positive constant corresponding  
to the reciprocal of temperature in statistical physics. Here $\{e_l \}$ is the standard orthonormal basis in $\bbr^d$. The one-dimensional Brownian motions $W_t^l$ are i.i.d. and satisfy the mean zero and covariance relations:
\[ \bbe[W_t^l]= 0 \quad \mbox{for $l = 1, \cdots, d$} \quad  \mbox{and} \quad  \bbe[W_t^{l_1} W_t^{l_2}] = \delta_{l_1 l_2} t, \quad 1 \leq  l_1, l_2 \leq d. \]

Note that system \eqref{A-1} is expected to have a local relaxation dynamics which means that $X^i_t$ relaxes toward the local weighted average ${\bar X}_t^*$, and this local weighted average tends to the global consensus state. This algorithm is an improvement upon those proposed in \cite{C-C-T-T, P-T-T-M} in that it is more suitable for higher-dimensional optimization problems, since its convergence conditions on system parameters are expected to be {\it independent} of the dimensionality $d$ due to the use of {\it componentwise} geometric Brownian motion.

We denote by $\Omega$ the underlying sample space for model \eqref{A-1}, and assume that the objective function $L$ is locally Lipschitz continuous. For the optimization algorithm \eqref{A-1}, we are interested in the following two questions: 
\begin{quote}
\begin{itemize}
\item
(Question I):~Does the $N$-state ensemble $\{ X_t^i \}$ exhibit a global consensus? i.e., for a.s. $\omega\in\Omega$, is there a global consensus state $X_\infty(\omega) \in \bbr^d$ such that 
\[ \lim_{t \to \infty} | X_t^i (\omega)- X_\infty (\omega) | = 0, \quad i = 1, \cdots, N, \]
where $| \cdot | := \| \cdot \|_{\ell^2}$ is the standard $\ell^2$-norm in $\bbr^d$ ? 

\vspace{0.1cm}

\item
(Question II):~If the answer to the first problem is positive, then under what condition  the consensus state is a good approximation of the global minimum $X^*$ of $L$? 
\end{itemize}
\end{quote}

In \cite{C-J-L-Z}, the authors conducted a convergence analysis via the Fokker-Planck equation, which can be deduced from \eqref{A-1} in the mean-field limit $N \to \infty$, and  showed that the global consensus state lies in a ${\mathcal O}(1/\beta)$-neighborhood of the global minimum under suitable assumptions on system parameters which are {\it independent} of dimension $d$ and initial data for $\beta >>1$. Since the Fokker-Planck equation is {\it not} the original model, thus this convergence result, although sheds lights on the convergence property of the original CBO model, it does not imply the latter.   The purpose of this article is to conduct the convergence analysis of \eqref{A-1} {\it directly}. 

Toward this goal, we first rewrite the relaxation term $X^i_t - {\bar X}_t^*$ \eqref{A-1} in consensus form: 
\begin{equation} \label{A-2}
\displaystyle dX^i_t = \lambda \sum_{k = 1}^{N} \psi^k_t (X^k_t - X^i_t) dt  + \sigma \sum_{k = 1}^{N} \sum_{l=1}^{d} \psi^k_t (x^{k,l}_t - x^{i,l}_t) dW^l_t e_l, \quad t > 0,
\end{equation}
where $\psi_t^k := \psi^k({\mathcal X}, t)$ is the communication weight function:
\begin{equation} \label{A-3}
\psi^k_t  :=  \frac{e^{-\beta L(X^k_t)}}{\sum_{l=1}^{N} e^{-\beta L(X^l_t)}}, \quad t \geq 0, \quad  k = 1, \cdots, N.
\end{equation}
Next, we return to the discrete-time dynamics associated with the continuous model \eqref{A-2}-\eqref{A-3}. We set a time-step $h := \Delta t$ and state at discrete time $t = nh$:
\[ X^i_n := X^i_{nh}, \qquad x_n^{i,l} := x_{nh}^{i,l}, \qquad  \psi^i_n := \psi^i_{nh}, \quad n = 0, 1, \cdots. \]
Then, the discrete consensus-based optimization model reads as follows:
\begin{equation} \label{A-4}
\begin{cases} 
\displaystyle X^i_{n +1}= X^i_n +   \lambda h \sum_{k = 1}^{N} \psi^k_n (X_n^k  - X_n^i) + \sigma \sqrt{h} \sum_{k = 1}^{N} \sum_{l=1}^{d} \psi_n^k (x_n^{k,l} - x_n^{i,l})  Z^l_n e_l, \quad n \geq 0, \\
\displaystyle \psi^k_n :=  \frac{e^{-\beta L(X_n^k)}}{\sum_{i=1}^{N} e^{-\beta L(X_n^i)}}, \quad \quad i = 1, \cdots, N,
\end{cases}
\end{equation}
where the random variables $\{ Z^l_n \}_{n, l}$ are i.i.d  standard normal distributions with $Z^l_n~\sim~{\mathcal N}(0, 1)$. \newline

We summarize the two main results of this paper now. First, we are concerned with the emergence of global consensus to the continuous and discrete models \eqref{A-1} and \eqref{A-4}, respectively. Since the analysis for the discrete model \eqref{A-4} is almost parallel to the analysis for the continuous one \eqref{A-1}, we mainly focus on the continuous model in what follows. To motivate the dynamic properties of \eqref{A-1} or \eqref{A-2}, we consider the deterministic counterpart:
\[ \frac{dX^i_t}{dt} = \lambda \sum_{k = 1}^{N} \psi^k_t (X^k_t - X^i_t). \]
In this case, it is easy to see that the time-dependent convex hull generated by $N$ sample points $X_t^1, \cdots, X_t^N$ in $\bbr^d$ is contractive (see Lemma \ref{L2.1}). Moreover, due to the special structure of the communication weight $\psi_t^k$:
\begin{align}
\begin{aligned} \label{A-5}
&(i)~\psi_t^k \geq 0, \quad 1 \leq k \leq N, \quad  \sum_{k=1}^{N} \psi_t^k = 1 \quad \mbox{for all $t \geq 0$}, \\
&(ii)~\mbox{Dependence only on the state of source sample point (independent of $i$ in \eqref{A-2})},
 \end{aligned}
 \end{align}
 the difference $X_t^i - X_t^j$ satisfies a system of ordinary differential equations:
\[ \frac{d}{dt}  (X^i_t - X^j_t) = -\lambda (X^i_t - X^j_t), \quad t > 0, \]
which has the analytic solution:
\[  (X^{i} - X^j) (t) = e^{-\lambda t} (X^{i}_0 - X^j_0), \quad t \geq 0. \]
On the other hand, maximal and minimal values of the component state of $X^i_t$ are monotone so that they converge to the same value. Hence, we can show that the state $X^i_t$ tends to the unique global consensus state $X_\infty$ independent of $i$ for any initial state (see Theorem \ref{T3.1}). 

Next, we return to the stochastic model with $\sigma > 0$. In this case, due to the white noise effect, the convex hull spanned by the state vector $X_t^i$ is not contractive any more. However, fortunately the $l$-th components of the relative state difference $x_t^{i,l} - x_t^{j,l}$ satisfies the geometric Brownian motion:
\[ d (x_t^{i,l} - x_t^{j,l}) = -\lambda (x_t^{i,l} - x_t^{j,l}) dt - \sigma (x_t^{i,l} - x_t^{j,l}) dW^l_t, \quad t > 0. \]
Then we use stochastic calculus to get the exact solution:
\[ x_t^{i,l} - x_t^{j,l} =  (x^{i,l}_0 - x^{j,l}_0) \exp \Big[  - \Big( \lambda  + \frac{\sigma^2}{2} \Big) t  + \sigma W^l_t \Big], \quad t \geq 0. \]
This yields the almost sure convergence of the relative state differences:
\[ \lim_{t \to \infty} |x^{i,l}_t - x^{j,l}_t| = 0, \quad \mbox{a.s.} \]
Similar analysis can be performed for the discrete algorithm \eqref{A-4} (see Theorem \ref{T3.3}).
Moreover, under the condition $2\lambda>\sigma^2$--which is {\it independent of the dimensionalizty} $d$,  we can show that there exists a random vector $X_\infty$ which is the almost-sure limit of the $X_t^i$'s (see Lemma \ref{L4.1}). Thus we answered the first posed question affirmatively for the continuous and discrete models. 

Second, we deal with Question II on whether the global consensus state $X_\infty$ is close to the global minimum of $L$ or not. Under suitable assumptions on system parameters $\lambda, \sigma$ and initial data such that  $X_0^i \sim X^{in}$~for some random variable $X^{in}$, we derive 
\[
\mathbb Ee^{-\beta L(X_\infty)}  \geq \varepsilon\mathbb E e^{-\beta L(X^{in})}, \quad \mbox{or} \quad -\frac{1}{\beta}\log \mathbb E e^{-\beta L(X_\infty)}\leq-\frac{1}{\beta}\log \mathbb E e^{-\beta L(X^{in})}-\frac{1}{\beta}\log\varepsilon.
\]
If the global minimizer $X^*$ of $L$ is contained in $\operatorname{supp}(\operatorname{law}(X^{in}))$, then Laplace's principle yields the desired estimate (see Theorem \ref{T4.1}):
\[
\operatorname{ess~inf}_{\omega\in\Omega} L(X^\infty(\omega))\leq L_m+ {\mathcal O}\Big(\frac{1}{\beta}\Big)\quad\mbox{for}\quad \omega \in \Omega,~~\beta \gg 1.
\]

Collective behaviors of agent-based models have been a hot topic in applied mathematics, control theory and related areas in recent years, see for example several survey articles \cite{A-B, A-B-F, C-H-L, P-R-K, V-Z} and related literature \cite{C-S, F-H-J, H-L, H-L-L,K-C-B-F-L, Ku1, Ku2, M-T, Pe}. \newline

The rest of this paper is organized as follows. In Section \ref{sec:2}, we provide preliminary materials on the deterministic analogs of the continuous and discrete algorithms \eqref{A-1} and \eqref{A-4}.  In Section \ref{sec:3}, we study the emergence of global consensus states for the continuous and discrete consensus models. In Section \ref{sec:4}, we prove the convergence of the global consensus state toward the global minimum of $L$ as $\beta \to \infty$ only for the continuous algorithm. The corresponding convergence analysis for the discrete model seems to be very challenging, thus  will be left for a future work. In Section \ref{sec:5}, we provide several numerical examples and compare them with our analytical results. Finally, Section \ref{sec:6} is devoted to a brief summary of our main results and discussion on some remaining problems to be investigated in future study. \newline

\noindent {\bf Notation}. For a random variable $Z \in \bbr$ on the probability space $(\Omega, {\mathcal F}, \bbp)$, we denote its mean by $\bbe Z$ or $\bbe[Z]$ interchangeably, and the function space ${\mathcal C}_b^2(\bbr^d)$ denotes the collection of all ${\mathcal C}^2(\bbr^d)$ functions with bounded derivatives up to second-order. 

\section{Preliminaries} \label{sec:2}
\setcounter{equation}{0}
In this section, we provide several preliminaries on the deterministic counterparts to \eqref{A-2}-\eqref{A-3} and \eqref{A-4} with $\sigma = 0$, respectively. \newline

Consider the continuous model:
\begin{equation} \label{B-1}
\begin{cases} 
\displaystyle \frac{dX^i_t}{dt} = \lambda \sum_{k = 1}^{N} \psi^k_t (X^k_t - X^i_t), \quad t > 0, \\
\displaystyle  \psi^k_t \geq 0, \quad \sum_{k=1}^{N} \psi^k_t = 1, \quad i = 1, \cdots, N,
\end{cases}
\end{equation}
and the discrete model:
\begin{equation} \label{B-2}
\begin{cases}
\displaystyle X^i_{n +1} = X^i_{n} +   \lambda h \sum_{k = 1}^{N} \psi^k_n (X^k_{n}- X^i_{n}), \quad n = 0, 1, \cdots, \\
\displaystyle \psi^k_n \geq 0, \quad \sum_{k=1}^{N} \psi^k_n = 1, \quad i = 1, \cdots, N.
\end{cases}
\end{equation}
Next, we study basic properties of the deterministic models \eqref{B-1} and \eqref{B-2} before moving to the stochastic ones.
\subsection{Deterministic continuous algorithm} \label{sec:2.1}
Let  ${\mathcal X}_t:= (X_t^1, \cdots, X^N_t) \in \bbr^{Nd}$ be a solution to \eqref{B-1}. For $t > 0$ and $l \in \{1, \cdots, d \}$, we introduce two extreme functions $ \underline{x}^l,~{\bar x}^l$ and component-diameter functional ${\mathcal D}^l({\mathcal X}_t)$:
\[  \underline{x}_t^l := \min_{1 \leq j \leq N} x_t^{j,l}, \qquad {\bar x}_t^l := \max_{1 \leq j \leq N} x_t^{j,l}, \qquad {\mathcal D}^l({\mathcal X}_t) := {\bar x}_t^l - \underline{x}_t^l. \]
Note that trajectories of $ \underline{x}^l$ and ${\bar x}^l$ are Lipschitz continuous, thus they are differentiable almost everywhere in $t\in(0, \infty)$. 
\begin{lemma} \label{L2.1}
Let $ {\mathcal X}_t = (X^1_t, \cdots, X^N_t)$ be a solution to \eqref{B-1} with the initial data ${\mathcal X}_0$.  Then, the following assertions hold.
\begin{enumerate}
\item
The extreme functions $\underline{x}^l$ and ${\bar x}^l$ are monotonically increasing and decreasing, respectively:
\[  \underline{x}_t^l \geq \underline{x}_s^l \quad \mbox{and} \quad {\bar x}_t^l \leq {\bar x}_s^l, \quad \mbox{for~~$t \geq s$}. \]
\item
The component diameter functional ${\mathcal D}^l({\mathcal X})$ is non-increasing in $t$:
\[ {\mathcal D}^l({\mathcal X}_t) \leq {\mathcal D}^l({\mathcal X}_{0}), \quad t \geq 0. \]
\end{enumerate}
\end{lemma}
\begin{proof} (i) Note that each component of \eqref{B-1} satisfies the same form of equations. Thus, it suffices to check one particular component. Consider the $l$-th component of system \eqref{B-1}:
\begin{equation} \label{B-3}
\frac{dx_t^{i,l}}{dt}  = \lambda \sum_{k=1}^{N} \psi_t^k (x_t^{k,l} - x_t^{i,l}).
\end{equation}
Now, we choose extreme indices $i_t$ and $j_t$ such that
\[ x_t^{i_t, l} = \underline{x}_t^l \quad \mbox{and} \quad x_t^{j_t, l}  = \bar{x}_t^l. \]

\noindent $\bullet$~Case A (Increasing property of $\underline{x}_t^l$): At a differentiable point $t$ of $\underline{x}_t^l$, it follows from \eqref{B-3} that 
\begin{equation} \label{B-4}
\frac{dx_t^{i_t, l}}{dt} = \lambda \sum_{k=1}^{N} \psi_t^k (x_t^{k,l} - x_t^{i_t, l}) \geq 0,
\end{equation}
where we used
\[  x_t^{k,l} - x_t^{i_t, l} \geq 0. \]
Then,  the Lipschitz continuity of $\underline{x}_t^l$ and \eqref{B-4} imply the non-increasing property of the lower envelope $\underline{x}_t^l$. \newline

\noindent $\bullet$~Case B (Decreasing property of $\bar{x}_t^l$): Similar to Case A, one has
\begin{equation*} \label{B-5}
\frac{dx^{j_t, l}}{dt} =  \lambda \sum_{k=1}^{N} \psi_t^k (x_t^{k,l} - x_t^{j_t, l}) \leq 0.
\end{equation*}
Finally, one combines Case A and Case B to see the non-increasing property of ${\mathcal D}^l({\mathcal X}_t)$: for $t \geq s$,
\[ {\mathcal D}^l({\mathcal X}_t) = {\bar x}_t^l - \underline{x}_t^l \leq {\bar x}_s^l - \underline{x}_s^l =  {\mathcal D}^l({\mathcal X}_s). \]
This yields the desired estimate.
\end{proof}
\begin{remark} \label{R2.1} Below, we comment two remarks for Lemma \ref{L2.1}. \newline

\noindent 1. The result of Lemma \ref{L2.1} implies that the convex hull of the set $\{X_t^i \} $ is non-increasing along the flow \eqref{B-1}. More precisely, set  
\[ C_t:= \mbox{convex hull} \{ X^1_t, \cdots, X^N_t \}. \]
Then, one has
\[ C_t \subset C_s, \quad \mbox{for $t \geq s$}. \]
\noindent 2. Using similar arguments, one can also show that the mixed norm $\| {\mathcal X}_t \|_{2,\infty}$:
\[| \| {\mathcal X}_t \|_{2,\infty} := \max_{1\leq i \leq N} \|X^i_t \|_{\ell^2} \]
is non-increasing in $t$:
\[  \| {\mathcal X}_t \|_{2,\infty} \leq \| {\mathcal X}_s \|_{2,\infty}, \quad \mbox{for $t \geq s$}. \]
\end{remark}
\subsection{Deterministic discrete algorithm} \label{sec:2.2}
Let  ${\mathcal X}_n:= (X^1_n, \cdots, X^N_n) \in \bbr^{Nd}$ be a state vector to system \eqref{B-2}, and  for $n \geq 0$ and $l \in \{1, \cdots, d \}$, set
\[ \underline{x}^l_n := \min_{1 \leq i \leq N} x^{i,l}_n \quad \mbox{and} \quad \overline{x}^l_n := \max_{1 \leq i \leq N} x^{i,l}_n. \]
Then, similar to Lemma \ref{L2.1}, one has the discrete counterpart for Lemma \ref{L2.1}.
\begin{lemma} \label{L2.2}
Let $ {\mathcal X}_n = (X^1_n, \cdots, X^N_n)$ be a state  to \eqref{B-2} with the initial data ${\mathcal X}_0$. Then, the following assertions hold.
\begin{enumerate}
\item
For each $k \in \{1, \cdots, N \}$, $ \underline{x}^l_n$ and $\overline{x}^l_n$ are monotonically increasing and decreasing, respectively:
\[  \underline{x}^l_n  \geq  \underline{x}^l_m \quad \mbox{and} \quad \overline{x}^l_n \leq \overline{x}^l_m, \quad \mbox{for~~$n \geq m$}. \]
\item
For $l \in \{1, \cdots, N \}$, the component diameter functional ${\mathcal D}^l({\mathcal X}_n)$ is non-increasing in $n$:
\[ {\mathcal D}^l({\mathcal X}_n )\leq {\mathcal D}^l({\mathcal X}_{0}), \quad n \geq 0. \]
\end{enumerate}
\end{lemma}
\begin{proof} Basically, we use the same arguments as in Lemma \ref{L2.1}. \newline

\noindent (i)  Let $i$ and $j$ be two indices such that 
\[ \underline{x}^l_n := x^{i,l}_n \quad \mbox{and} \quad \overline{x}^l_n := x^{j,l}_n. \]
Then for such $i$, by $\psi_n^k (x^{k,l}_n - x^{i,l}_n) \geq 0$, one has
\[ x^{i,l}_{n+1} = x^{i,l}_n + \lambda h \sum_{k=1}^{N} \psi^k_n (x^{k,l}_n - x^{i,l}_n)  \geq x^{i,l}_n = \underline{x}^l_n. \]
This implies 
\begin{equation} \label{B-5}
 \underline{x}^l_{n+1} \geq \underline{x}^l_{n}.   
\end{equation} 
Hence $\underline{x}^l_n$ is non-decreasing in $n$.  Similarly, one has 
\[  x^{j,l}_{n+1} = x^{j,l}_{n} + \lambda h \sum_{k=1}^{N} \psi_n^k (x_n^{k,l} - x_n^{j,l})  \leq x_n^{j,l} = \overline{x}_n^l. \]
This yields
\begin{equation} \label{B-6}
\overline{x}^l_{n+1} \leq  \overline{x}_n^l. 
\end{equation}
\noindent (ii)~We combine \eqref{B-5} and \eqref{B-6} to get 
\[  {\mathcal D}^l({\mathcal X}_{n +1}) = \overline{x}^l_{n+1}  - \underline{x}^l_{n+1} \leq  \overline{x}^l_n -  \underline{x}^l_n =  {\mathcal D}^l({\mathcal X}_n). \]
\end{proof}
\begin{remark}
The result of Lemma \ref{L2.2} yields
\[  \min_{1 \leq k \leq N} x_0^{k,l} \leq x_n^{i,l} \leq  \max_{1 \leq k \leq N} x_0^{k,l}, \quad i = 1, \cdots, N,~~ n \geq 0. \]
\end{remark}

\section{Emergence of global consensus} \label{sec:3}
\setcounter{equation}{0}
In this section, we study the emergence of global consensus to systems \eqref{B-1} and \eqref{B-2} based on the following two steps:
\begin{itemize}
\item
Step A: We first derive an explicit formula for the state difference $X^i_t - X^j_t$, and then by using this formula, we show that the relative state differences tend to zero exponentially fast.

\vspace{0.1cm}
\item
Step B: For each component, we show that the maximal and minimal values are monotonically decreasing and increasing respectively over time so that as time tends to infinity, all extremal states tend to the same value. Then, together with the result of Step A, we can see that all states converge to the same global consensus state $X_\infty$ independent of particle number $i$. 
\end{itemize}

\subsection{Stochastic continuous algorithm} \label{sec:3.1}
Consider the continuous algorithms for $X^i_t$ and $X^j_t$: 
\begin{equation}
\begin{cases} \label{C-1}
\displaystyle dX^i_t = \lambda \sum_{k = 1}^{N} \psi^k_t (X^k_t - X^i_t) dt  + \sigma \sum_{k = 1}^{N} \sum_{l=1}^{d} \psi^k_t(x^{k,l}_t - x^{i,l}_t) dW^l_t e_l, \quad t > 0, \\
\displaystyle dX^j_t = \lambda \sum_{k = 1}^{N} \psi^k_t (X^k_t - X^j_t) dt  + \sigma \sum_{k = 1}^{N} \sum_{l=1}^{d} \psi^k_t(x^{k,l}_t - x^{j,l}_t) dW^l_t e_l, \quad t > 0, \\
\end{cases}
\end{equation}
subject to the initial data:
\begin{equation} \label{C-2}
X^i_t \Big|_{t = 0} = X^{i}_0,  \quad X^j_t \Big|_{t = 0} = X^{j}_0.
\end{equation}
First, we use the unit sum  condition \eqref{A-5} of $\psi^k_t$'s to get 
\begin{equation} \label{C-3}
\sum_{k = 1}^{N} \psi^k_t (X^k_t - X^i_t) - \sum_{k = 1}^{N} \psi^k_t (X^k_t - X^j_t)  = -\Big( \sum_{k = 1}^{N} \psi^k_t \Big) (X^i_t - X^j_t) = -(X^i_t - X^j_t).
\end{equation}
Note that the dependence on $\psi_t^i$ disappears on the R.H.S. of \eqref{C-3}. Thus, one uses \eqref{C-3} to see that $X^i_t - X^j_t$ satisfies 
\begin{equation} \label{C-4}
d (X^i_t - X^j_t) = -\lambda (X^i_t - X^j_t) dt - \sigma \sum_{l=1}^{d} (x^{i,l}_t - x^{j,l}_t) dW^l_t e_l.
\end{equation}
Next, we provide the emergence of global consensus to the deterministic model \eqref{C-4} with $\sigma = 0$.  For a given configuration process ${\mathcal X}_t$, we set
\[ {\mathcal D}({\mathcal X}_t) := \max_{1 \leq i, j \leq N} |X_t^i - X_t^j |. \]
\begin{theorem} \label{T3.1}
Let $\{X^i \}$ be a solution to \eqref{C-1} - \eqref{C-2} with $\sigma = 0$. Then, the following two assertions hold.
\begin{enumerate}
\item
The diameter ${\mathcal D}({\mathcal X})$ decays to zero exponentially fast:
\[ \mathcal D({\mathcal X}_t) \leq e^{- \lambda t} \mathcal D({\mathcal X}_0), \quad t \geq 0. \] 
\item
There exits a unique global consensus state $X_{\infty} = (x^1_\infty, \cdots, x^d_\infty) \in \bbr^d$ such that  for all $i = 1, \cdots, N$,
\[ \lim_{t \to \infty} X^i_t = X_\infty. \]
\end{enumerate}
\end{theorem}
\begin{proof}  (i)~It follows from \eqref{C-4} that 
\[ \frac{d}{dt}  (X^i_t - X^j_t) = -\lambda (X^i_t - X^j_t), \quad t > 0. \]
This yields 
\[  |X_t^{i} - X_t^j| = e^{-\lambda t} |X^{i}_0 - X^j_0|, \quad t \geq 0. \]
Again, by taking maximum over all the indices $i$ and $j$, one gets the desired exponential decay of ${\mathcal D}({\mathcal X}_t)$. \newline

\noindent (ii)~For each $l \in \{1, \cdots, d\}$, we claim that the extreme functions ${\bar x}_t^l$ and $\underline{x}_t^l$ converge to the same value $x^l_\infty$ so that all the other state $x^{i,l}$ should converge to the same value $x^{l}_\infty$, because the component diameter  shrinks to zero asymptotically. In the course of proof of (i), we showed that ${\bar x}_t^l$ is non-increasing and bounded by ${\bar x}_0^l$. Thus, it should converge to ${\bar x}^l_\infty$. Similarly, $\underline{x}^l$ should converge to $\underline{x}^l_\infty$. Then, it is easy to see that ${\bar x}^l_\infty = \underline{x} ^l_\infty$ due to the exponential decay of the component diameter.
\end{proof}
\begin{remark} Note that the explicit form of $\psi_t^i$ does not appear in the decay estimate of the diameter.
\end{remark}

\bigskip

On the other hand, it follows from \eqref{C-4} that $x_t^{ij, l} := x_t^{i,l} - x_t^{j,l}$ satisfies 
\begin{equation} \label{C-5}
\begin{cases}
d x^{ij,l}_t  = -\lambda x^{ij,l}_t dt - \sigma x^{ij,l}_t dW^l_t, \quad t > 0, \\
x^{ij,l}_t \Big|_{t = 0} = x^i_0 - x^j_0.
\end{cases}
\end{equation}
Now, we apply Ito's formula for $\ln x^{ij,l}_t$ using \eqref{C-5} to see
\begin{equation} \label{C-6}
d \ln x^{ij,l}_t = \frac{dx^{ij,l}_t}{x^{ij,l}_t} -\frac{1}{2 (x^{ij,l}_t)^2} dx^{ij,l}_t \cdot dx^{ij,l}_t =  
- \Big( \lambda + \frac{\sigma^2}{2}  \Big) dt + \sigma dW^l_t.
\end{equation}
Integrating the above relation \eqref{C-6} gives
\begin{equation} \label{C-11}
x^{ij,l}_t = x^{ij,l}_0 \exp \Big[  - \Big( \lambda  + \frac{\sigma^2}{2} \Big) t  + \sigma W^l_t \Big], \quad t \geq 0. 
\end{equation}
Then, the explicit formula \eqref{C-11} yields the following result. 
\begin{theorem} \label{T3.2}
Let $\{ X^i_t \}$ be a solution process of \eqref{B-1}. Then, for $i \not = j  = 1, \cdots N$ and $k \in \{1, \cdots, d \}$, 
\[ \lim_{t \to \infty} |x^{i,k}_t - x^{j,k}_t| = 0, \quad \mbox{a.s.} \quad \mbox{and} \quad \lim_{t \to \infty} {\mathbb P} \Big( |x^{i,k}_t - x^{j,k}_t|^2> \varepsilon \Big) = 0, \quad \mbox{for any $\varepsilon > 0$}.
\]
\end{theorem}
\begin{proof}
The proof is essentially the same as in Theorem 3.1 of \cite{A-H}. However, for reader's convenience, we briefly sketch the proof here. \newline

\noindent (i)~Recall the law of iterated logrithm of the Brownian motion:
\begin{equation*} \label{C-12}
 \limsup_{t \to \infty} \frac{|W^k_t|}{\sqrt{2t \log \log t}} = 1, \quad \mbox{a.s.}, 
\end{equation*}
and note that for $t \gg 1$, the linear negative term $-\Big( \lambda  + \frac{\sigma^2}{2} \Big) t $  in $t$ is certainly dominant compared to the Brownian term $\sigma W^k_t$ which grows with a rate of at most $t^{\frac{1}{2} + }$. Thus, the trajectory $|x^{i,k}_t - x^{j,k}_t|$ tends to zero almost surely as $t \to \infty$.  \newline

\noindent (ii)~Since a.s. convergence implies convergence in probability, the result follows from (i).
\end{proof}

\subsection{Stochastic discrete algorithm} \label{sec:3.2} Consider the discrete algorithms: for $i, j = 1, \cdots, N$, 

\begin{equation} \label{C-13}
\begin{cases}
\displaystyle X^i_{n +1} = X^i_n  + \lambda h \sum_{k = 1}^{N} \psi_n^k (X_n^k - X_n^i) + \sigma \sqrt{h} \sum_{k = 1}^{N} \sum_{l=1}^{d} \psi_n^k (x_n^{k,l} - x_n^{i,l})  Z^l_n e_l,  \\
\displaystyle X^j_{n +1} = X^j_n  + \lambda h \sum_{k = 1}^{N} \psi_n^k (X_n^k - X_n^j) + \sigma \sqrt{h} \sum_{k = 1}^{N} \sum_{l=1}^{d} \psi_n^k (x_n^{k,l} - X_n^{j,l})  Z^l_n e_l,
\end{cases}
\end{equation}
subject to the initial data:
\begin{equation} \label{C-14}
X_n^i \Big|_{n =0} = X^{i}_0,  \quad X_n^j \Big|_{n= 0} = X^{j}_0.
\end{equation}
Here the random variables $\{ Z^l_n \}$ are i.i.d. and satisfy $Z^l_n \sim {\mathcal N}(0, 1)$. \newline

Note that 
\[ \sum_{k = 1}^{N} \psi_n^k (X_n^k - X_n^i)  - \sum_{k = 1}^{N} \psi_n^k (X_n^k - X_n^j) =  - (X_n^i - X_n^j).   \]
Thus, one has
\begin{equation} \label{C-15}
x^{i,k}_{n +1}  - x^{j,k}_{n +1} = \Big(1-\lambda h - \sigma \sqrt{h}  Z^k_n  \Big) (x^{i,k}_n  - x_n^{j,k}). 
\end{equation}
Based on the above explicit recursive relation, we have the following emergent dynamics.
\begin{theorem} \label{T3.3}
Suppose that parameters satisfy
\[ \sigma = 0, \quad \lambda >0 \quad \mbox{and} \quad 0<h<\frac{1}{\lambda}. \]
Then, for any solution $\{X_n^i \}$ to \eqref{C-13} - \eqref{C-14}, the following two assertions hold.
\begin{enumerate}
\item
The diameter ${\mathcal D}({\mathcal X}_n)$ decays to zero exponentially fast:
\[ \mathcal D({\mathcal X}_n) \leq e^{- \lambda nh} \mathcal D({\mathcal X}_0), \quad n = 0, 1, \cdots.  \] 
\item
There exits a unique global consensus state $X_{\infty} = (x^1_\infty, \cdots, x^d_\infty) \in \bbr^d$ independent of $i$ such that  for all $i = 1, \cdots, N$,
\[ \lim_{n \to \infty} X^i_n = X_\infty. \]
\end{enumerate}
\end{theorem}
\begin{proof}
(i)~It follows from \eqref{C-15} that 
\begin{align*}
\begin{aligned} \label{C-16}
X^i_{n+1} -X^j_{n+1} &= X^i_n - X^j_n -\lambda h \Big( \sum_{l=1}^{N} \psi_n^l \Big) (X^i_n - X^j_n) = \Big(1 -\lambda h \Big) (X^i_n - X^j_n).
\end{aligned}
\end{align*}
This implies
\[ X_n^i -X_n^j  = \Big(1 -\lambda h \Big)^n (X_0^i - X_0^j). \]
Thus, we have the desired estimate:
\[  |X_n^i -X_n^j|  = \Big(1 -\lambda h \Big)^n |X_0^i - X_0^j| \leq e^{-\lambda h n} |X_0^i - X_0^j|, \quad n = 0, 1, \cdots. \]

\noindent (ii)~We use the same argument in Theorem \ref{T3.1} to get the desired convergence.
\end{proof}

Now, we return to the stochastic version with $\sigma > 0$ in the following theorem. 

\begin{theorem} \label{T3.4}
Let $\{X_n^i \}$ be a solution to \eqref{C-13} - \eqref{C-14}. Then, we have the following stochastic consensus:
\begin{enumerate}
\item
(Weak stochastic consensus I):~Suppose that $h$ and systems parameters satisfy 
\[ \sigma >0, \quad  \lambda > 0 \quad \mbox{and} \quad  0 <  h < \frac{1}{\lambda}.  \]
Then, for $i, j = 1, \cdots, N$, one has
\[ \Big|\mathbb E[X_n^i-X_n^j] \Big|  \leq e^{-\lambda nh} \Big| \mathbb E[X_0^i-X_0^j] \Big|. \]
\item
(Weak stochastic consensus II):~Suppose that $h$ and systems parameters satisfy 
\[ \sigma > 0, \quad   2\lambda > \sigma^2 \quad \mbox{and} \quad 0 <  h < \frac{2\lambda - \sigma^2}{\lambda}.  \]
Then, for $i, j = 1, \cdots, N$, one has the following exponential decay estimate:
\[  \mathbb E|X_n^i-X_n^j|^2 \leq e^{-nhm} \bbe |X^i_0 - X^j_0|^2, \]
where $m = m(\lambda,h,\sigma)$ is defined as follows:
\[  m(\lambda, h, \sigma) := 2\lambda - \lambda^2 h -\sigma^2.   \]
\item
(Strong stochastic consensus):~Suppose that $h$ and system parameters satisfy
\[ \sigma >0, \quad  2 \lambda > \sigma^2 \quad \mbox{and} \quad 0 <  h < \frac{2\lambda - \sigma^2}{\lambda}. \]
Then, for $i, j = 1, \cdots, N$, one has
\[
 |X^i_n - X^j_n|  \leq  e^{-n  Y_n} |X^i_0 - X^j_0|, \quad \mbox{a.s.}~\omega \in \Omega,  
\]
where $Y_n$ is a random variable satisfying 
\[ \lim_{n \to \infty} Y_n(\omega) = \frac{hm}{2} := \frac{h}{2} (2\lambda-\lambda^2 h-\sigma^2) > 0, \quad \mbox{a.e.}~\omega \in \Omega. \]
\end{enumerate}
\end{theorem}
\begin{proof}
(i)~It follows from \eqref{C-15} that 
\begin{align}
\begin{aligned} \label{D-0-9}
&X^i_{n +1} - X^j_{n+1} \\
& \hspace{1cm} = X^i_n - X^j_n -\lambda h ( X^i_n - X^j_n) + \sigma \sqrt{h} (X^i_n - X^j_n) Z_n = \Big(1 -\lambda h + \sigma \sqrt{h} Z_n  \Big) (X^i_n - X^j_n).
\end{aligned}
\end{align}
Iterating the above recursive relation \eqref{D-0-9} gives
\begin{equation} \label{D-0-10}
X^i_n - X^j_n = \Pi_{\ell = 0}^{n-1} \Big(1- \underbrace{(\lambda h - \sigma \sqrt{h} Z_\ell)}_{=: \Delta_\ell}  \Big) (X^i_0 - X^j_0).
\end{equation}
Finally, we take expectation and absolute value on both sides of \eqref{D-0-10} using the independence of $Z_\ell$ and $X^i_0 - X^j_0$ to get
\[ \Big|\mathbb E[X_n^i-X_n^j] \Big| =(1-\lambda h)^n \Big| \mathbb E[X_0^i-X_0^j] \Big|\leq e^{-\lambda nh} \Big| \mathbb E[X_0^i-X_0^j] \Big|. \]
 \newline
\noindent (ii)~We take the absolute value of \eqref{D-0-10} and square of it to find
\begin{equation} \label{D-0-11}
|X^i_n - X^j_n|^2 = \Pi_{\ell = 0}^{n-1} (1- \Delta_\ell )^2 |X^i_0 - X^j_0|^2.
\end{equation}
Taking expectation of the above relation and using the independence of $\Delta_\ell$ and $|X^i_0 - X^j_0|$, one gets
\begin{equation} \label{D-0-12}
\bbe |X^i_n - X^j_n|^2  =  \Pi_{\ell = 0}^{n-1} \bbe[(1- \Delta_\ell )^2] \times \bbe[|X^i_0 - X^j_0|^2].  
\end{equation}
On the other hand, since $\{\Delta_\ell \}$ are i.i.d. and for each $\ell = 0, \cdots, n-1$,  one has 
\begin{align}
\begin{aligned}  \label{D-0-13}
\bbe[(1- \Delta_\ell )^2] &= 1-{\mathbb E} [2\Delta_\ell - \Delta_\ell^2 ] \\
&=1- {\mathbb E} \Big(2 \lambda h - 2\sigma \sqrt{h} Z_\ell -  (\lambda h)^2 + 2\sigma \lambda h^{\frac{3}{2}} Z_\ell - \sigma^2 h Z_\ell^2 \Big) \\
&=1- 2\lambda h + \lambda^2 h^2 + \sigma^2 h = 1 -h m(\lambda,h, \sigma) \geq 0.
\end{aligned}
\end{align}
Now, we combine \eqref{D-0-12} and \eqref{D-0-13} to get 
\[ \bbe |X^i_n - X^j_n|^2  = (1-hm)^{n}  \bbe|X^i_0 - X^j_0|^2\leq e^{-mnh}\bbe|X^i_0 - X^j_0|^2 .\]
So a sufficient condition for the exponential decay of $\bbe |X^i_n - X^j_n|^2$ is
\[
m(\lambda,h, \sigma)>0\quad \Longleftrightarrow \quad h < \frac{2\lambda-\sigma^2}{\lambda}.  
\]
\noindent (iii)~It follows from \eqref{D-0-11} and the inequality:
\[ (1 -\Delta_\ell)^2 = 1 - (2\Delta_\ell - \Delta_\ell^2) \leq e^{-(2\Delta_\ell - \Delta_\ell^2)} \]
that 
\begin{align}
\begin{aligned} \label{D-0-11}
|X^i_n - X^j_n|^2 &= \Pi_{\ell = 0}^{n-1} (1- \Delta(\ell) )^2 |X^i_0 - X^j_0|^2 \leq  \Pi_{\ell = 0}^{n-1} e^{-(2\Delta_\ell - \Delta_\ell^2)} |X^i_0 - X^j_0|^2 \\
&= \exp \Big[ -\sum_{\ell = 0}^{n-1} \Delta_\ell (2-\Delta_{\ell}) \Big] |X^i_0 - X^j_0|^2 \\
&=  \exp \Big[ -n \times \frac{1}{n} \sum_{\ell = 0}^{n-1} \Delta_\ell (2-\Delta_{\ell}) \Big] |X^i_0 - X^j_0|^2.
\end{aligned}
\end{align}
On the other hand, using \eqref{D-0-13} one has
\[ {\mathbb E} [2\Delta_\ell - \Delta_\ell^2] =h m. \]
Next, we use the strong law of large numbers to see
\[  Y_n :=  \frac{1}{n} \sum_{\ell = 0}^{n-1} \Delta_\ell (2-\Delta_{\ell})~~ \to~~ {\mathbb E} \Big[ \Delta_\ell (2-\Delta_{\ell}) \Big] = h m \quad \mbox{as 
$n \to \infty$~a.s.}  \]
where $Y_n$ is also a random variable since $\Delta_\ell$ is a random variable. 
In \eqref{D-0-11}, we use the above convergence to see that
\[
|X^i_n - X^j_n| \leq  \exp \Big[ -n \times \frac{1}{2n} \sum_{\ell = 0}^{n-1} \Delta_\ell (2-\Delta_{\ell}) \Big] |X^i_0 - X^j_0|
\]
and
\[
\frac{1}{2n} \sum_{\ell = 0}^{n-1} \Delta_\ell (2-\Delta_{\ell}) \to  \frac{hm}{2},\quad  n\longrightarrow\infty.
\]
\end{proof}
\begin{remark} If one uses the result (ii) and the Cauchy-Schwarz inequality, one can obtain
\[  \bbe |X^i_n - X^j_n| \leq \sqrt{\bbe[ |X^i_n - X^j_n|^2] } \leq  (1-hm)^{\frac{n}{2}} \sqrt{\bbe [ |X^i_0 - X^j_0|^2 ]}.   \]

\end{remark}
\section{Convergence analysis for continuous algorithm} \label{sec:4}
\setcounter{equation}{0}
In this section, we provide a convergence analysis for the continuous CBO algorithm using Ito's calculus. In previous section, we showed that the continuous algorithm admits a global consensus for any initial data. Thus, the natural question is whether this global consensus is a global minimum of $L$ or not. If the answer is affirmative, then under what condition such a coincidence will occur? This is the main concern of this section. \newline

Recall that $X_t^i$ satisfies 
\begin{equation} \label{D-1}
dX^i_t = -\lambda (X^i_t - {\bar X}_t^*)dt  + \sigma  \sum_{l=1}^{d} (x^{i,l}_t - {\bar x}_t^{*,l}) dW_t^l e_l,
\end{equation}
and we introduce an ensemble average:
\[
{\bar X}_t:=\frac{1}{N}\sum_{i=1}^N X_t^i = ({\bar x}_t^1, \cdots, {\bar x}_t^d).
\]
Next, we present three elementary lemmas to be crucially used in the proof of convergence analysis.
\begin{lemma}\label{L4.1}
Let $\{X_t^i\}_{1\leq i\leq N}$ be a solution to \eqref{D-1}. Then, the following estimates hold almost surely. 
\begin{eqnarray*}
&& (i)~|X_t^i-\bar X_t|^2=\sum_{l=1}^d (x_0^{i,l} -{\bar x}^l_0)^2\exp\left[ -\Big(2\lambda+\sigma^2\Big)t+2\sigma W_t^l\right].\\
&& (ii)~|\bar X_t-\bar X_t^*|^2\leq \max_{1\leq i\leq N}|X_t^i-\bar X_t|^2.\\
&& (iii)~\frac{1}{N}\sum_{i=1}^N |X_t^i-\bar X_t^*|^2\leq 2 \sum_{l=1}^d \left(\max_{1\leq i\leq N} (x_0^{i,l} -{\bar x}^l_0)^2\right)\exp\left[ -\Big(2\lambda+\sigma^2\Big)t+2\sigma W_t^l\right].
\end{eqnarray*}
\end{lemma}
\begin{proof}
(i)~It follows from \eqref{D-1} that 
\begin{equation}\label{D-5}
d{\bar X}_t=-\lambda({\bar X}_t-\bar X_t^*)dt+\sigma\sum_{l=1}^d( {\bar x}^l_t- {\bar x}_t^{*,l}) dW_t^l e_l.
\end{equation}
We subtract \eqref{D-1} from \eqref{D-5} to obtain
\begin{equation} \label{D-6}
d(X_t^i- {\bar X}_t)=-\lambda(X_t^i- {\bar X}_t)dt+\sigma\sum_{l=1}^d( x_t^{i,l}- {\bar x}^{l}_t) dW_t^l  e_l.
\end{equation}
The $l$-th component of \eqref{D-6} implies
\[
x_t^{i,l} -{\bar x}^{l}_t = (x_0^{i,l} -{\bar x}^l_0) \exp\left[ -\Big(\lambda+\frac{1}{2}\sigma^2\Big)t+\sigma W_t^l \right].
\]
This yields
\[
|X_t^i-\bar X_t|^2=\sum_{l=1}^d (x_0^{i,l} -{\bar x}^l_0)^2\exp\left[ -\Big(2\lambda+\sigma^2\Big)t+2\sigma W_t^l\right].
\]
\noindent (ii)~We use the triangle inequality and the Cauchy-Schwarz inequality to get
\begin{align*}
\begin{aligned}
|{\bar X}_t-\bar X_t^*|^2&=\left|\frac{\sum_{k=1}^Ne^{-\beta L(X_t^k)}( {\bar X}_t-X_t^k)}{\sum_{k=1}^Ne^{-\beta L(X_t^k)}}\right|^2\leq \left[\frac{\sum_{k=1}^Ne^{-\beta L(X_t^k)}| {\bar X}_t-X_t^k|}{\sum_{k=1}^Ne^{-\beta L(X_t^k)}}\right]^2\\
&\leq \frac{\sum_{k=1}^N e^{-\beta L(X_t^k)}| {\bar X}_t-X_t^k|^2}{\sum_{k=1}^Ne^{-\beta L(X_t^k)}}\leq \max_{1\leq k \leq N}| {\bar X}_t-X_t^k |^2.
\end{aligned}
\end{align*}
(iii)~Note that
\begin{align}\label{D-6-1}
\begin{aligned}
\frac{1}{N}\sum_{i=1}^N |X_t^i-\bar X_t^*|^2&=\frac{1}{N}\sum_{i=1}^N \left(|X_t^i-\bar X_t|^2+2(X_t^i-\bar X_t)\cdot(\bar X_t-\bar X_t^*)+|\bar X_t-\bar X_t^*|^2\right)\\
&=\frac{1}{N}\sum_{i=1}^N  |X_t^i-\bar X_t|^2 +|\bar X_t-\bar X_t^*|^2 \leq 2\max_{1\leq i\leq N} |X_t^i-\bar X_t|^2\\
&=2\max_{1\leq i\leq N} \left(\sum_{l=1}^d (x_0^{i,l} -{\bar x}^l_0)^2\exp\left[ -\Big(2\lambda+\sigma^2\Big)t+2\sigma W_t^l\right]\right)\\
&\leq 2 \sum_{l=1}^d \left(\max_{1\leq i\leq N} (x_0^{i,l} -{\bar x}^l_0)^2\right)\exp\left[ -\Big(2\lambda+\sigma^2\Big)t+2\sigma W_t^l\right],
\end{aligned}
\end{align}
where we used the inequalities from (i) and (ii).
\end{proof}

\begin{lemma}\label{L4.2}
Let $\{X_t^i\}_{1\leq i\leq N}$ be a solution to \eqref{D-1}. Then, the following estimates hold. 
\begin{eqnarray*}
&& (i)~\frac{1}{N}\sum_{i=1}^N \mathbb E|X_t^i-\bar X_t^*|^2\leq 2e^{-(2\lambda-\sigma^2)t} \sum_{l=1}^d \mathbb E \Big[ \max_{1\leq i\leq N} (x_0^{i,l} -{\bar x}^l_0)^2 \Big].\\
&& (ii)~ \mbox{If}~2\lambda>\sigma^2,~\mbox{then there exists a random vector}~X_\infty~ \mbox{such that}~ \\
&& \hspace{3.5cm} \lim\limits_{t\to\infty}  X_t^i=X_\infty~\mbox{a.s.},~1\leq i\leq N.
\end{eqnarray*}
\end{lemma}
\begin{proof}
\noindent (i)~We take expectation on both sides of \eqref{D-6-1} to get
\begin{align*}
\begin{aligned}
\frac{1}{N}\sum_{i=1}^N \mathbb E|X_t^i-\bar X_t^*|^2&\leq 2 \sum_{l=1}^d \left(\mathbb E\max_{1\leq i\leq N} (x_0^{i,l} -{\bar x}^l_0)^2\right)\mathbb E\exp\left[ -\Big(2\lambda+\sigma^2\Big)t+2\sigma W_t^l\right]\\
&\leq 2e^{-(2\lambda-\sigma^2)t} \sum_{l=1}^d \left(\mathbb E\max_{1\leq i\leq N} (x_0^{i,l} -{\bar x}^l_0)^2\right),
\end{aligned}
\end{align*}
where we used
\[
\exp\left(2\sigma W_t^l\right)\sim\operatorname{Lognormal}(0,4\sigma^2t)\quad\Longrightarrow \quad \mathbb E\exp\left(2\sigma W_t^l\right)=\exp(2\sigma^2 t).
\]
(ii)~Note that equation $\eqref{D-1}$ is equivalent to the following integral relation: for $i=1,\cdots,N$ and $l=1,\cdots,d$, 
\begin{equation*}
x_t^{i,l}=x_0^{i,l}-\lambda\int_0^t (x_s^{i,l} - \bar x_s^{*,l}) ds + \sigma \int_0^t (x_s^{i,l} - \bar x_s^{*,l}) dW_s^l=:x_0^{i,l}-\lambda {\mathcal I}_{11} + \sigma {\mathcal I}_{12}.
\end{equation*}
Next, we show the a.s. convergence ${\mathcal I}_{11}$ and ${\mathcal I}_{12}$ separately.\newline

\noindent $\bullet$~Case A (Almost sure convergence of ${\mathcal I}_{11}$):~ By (iii), we have 
\[
|x_t^{i,l} - \bar x_t^{*,l}|\leq \sqrt{\sum_{i=1}^N |X_t^i-\bar X_t^*|^2}\leq \sqrt{2N \sum_{l=1}^d \left(\max_{1\leq i\leq N} (x_0^{i,l} -{\bar x}^l_0)^2\right)\exp\left[ -\Big(2\lambda+\sigma^2\Big)t+2\sigma W_t^l\right]}.
\]
This yields that there exist positive random functions $C_i = C_i(\omega),~i= 1,2$ such that 
\[ |x_t^{i,l} - \bar x_t^{*,l}|\leq C_1e^{-C_2t},  \quad \mbox{a.s.}~\omega \in \Omega, \] 
where $C_1$ and $C_2$ are positive constants. We set 
\[ {\mathcal J}_{11} := {\mathcal I}_{11} -\int_0^t  C_1e^{-C_2s} ds = \int_0^t \underbrace{\big(x_s^{i,l} - \bar x_s^{*,l} -C_1e^{-C_2s} \big)}_{\leq 0} ds. \]
Since the integrand is nonpositive a.s., $ {\mathcal J}_{11}$ is non-increasing in $t$ a.s.  \newline

On the other hand, note that 
\begin{align*}
\begin{aligned}
& {\mathcal J}_{11} = {\mathcal I}_{11} -\frac{C_1}{C_2}(1-e^{-C_2t}) \geq {\mathcal I}_{11} -\frac{2C_1}{C_2}+\frac{C_1}{C_2}(1-e^{-C_2t}) \\
& \hspace{0.5cm} = {\mathcal I}_{11} -\frac{2C_1}{C_2}+\int_0^t  C_1e^{-C_2s} ds =-\frac{2C_1}{C_2}+\int_0^t  \big(x_s^{i,l} - \bar x_s^{*,l} +C_1e^{-C_2s} \big) ds \geq -\frac{2C_1}{C_2}.
\end{aligned}
\end{align*}
Since ${\mathcal J}_{11}$ is monotone decreasing and bounded below along sample paths, one has
\[  \exists~\alpha = \lim_{t \to \infty} {\mathcal J}_{11}(t) = \lim_{t \to \infty}  \Big( {\mathcal I}_{11} -\int_0^t  C_1e^{-C_2s} ds \Big), \quad \mbox{a.s.} \]
This implies
\[ \lim_{t\to\infty} {\mathcal I}_{11}=\alpha+\frac{C_1}{C_2}, \quad \mbox{a.s.} \]

\noindent $\bullet$~Case B (Almost sure convergence of ${\mathcal I}_{12}$):~Note that ${\mathcal I}_{12}$ is martingale and its $L^2(\Omega)$-norm is uniformly bounded in $t$:
\begin{align*}
\begin{aligned}
\mathbb E \left[\int_0^t \Big(x_s^{i,l} - \bar x_s^{*,l} \Big) dW_s^l\right]^2 &=\mathbb E \int_0^t (x_s^{i,l} - \bar x_s^{*,l})^2 ds\leq \int_0^t\sum_{i=1}^N \mathbb E|X_s^i-\bar X_s^*|^2ds\\
&\leq 2N\left(\int_0^t e^{-(2\lambda-\sigma^2)s}ds\right) \sum_{l=1}^d \left(\mathbb E\max_{1\leq i\leq N} (x_0^{i,l} -{\bar x}^l_0)^2\right) \\
&\leq  \frac{2N}{2\lambda-\sigma^2} \sum_{l=1}^d \left(\mathbb E\max_{1\leq i\leq N} (x_0^{i,l} -{\bar x}^l_0)^2\right).
\end{aligned}
\end{align*}
In the second inequality we used (iv). Hence $\lim\limits_{t\to\infty} {\mathcal I}_{12}$ exists a.s. Now we have shown that for each $i=1,\cdots, N$, there exists some random variable $X^{i}_{\infty}$ such that 
\[ \lim\limits_{t\to\infty}X_t^i=X^{i}_{\infty} \quad \mbox{a.s.} \]
Since for any $1\leq i,j\leq N$,  
\[ \lim\limits_{t\to\infty}|X_t^i-X_t^j|=0, \quad \mbox{a.s.} \]
Hence, there exists $X_\infty$ such that 
\[ X^{i}_{\infty}=X^{j}_{\infty}=:X_\infty \quad \mbox{a.s.} \]
\end{proof}
\begin{lemma}\label{L4.3}
 Let $\{X_t^i\}_{1\leq i\leq N}$ be a solution to \eqref{D-1}. Then, the quadratic variation of $x^k_t$ and $x^l_t$ is given as follows.
\[ dx^k_t \cdot dx^l_t = \begin{cases}
\sigma^2 |x^{i,k}_t - {\bar x}_t^{*,k}|^2 dt, \quad & k = l, \\
0, \quad & k \not = l.
\end{cases} \]
\end{lemma}
\begin{proof} It follows from \eqref{D-1} that the $l$-th component of $X_t^{i}$  satisfies 
\begin{align}
\begin{aligned} \label{NE-1}
dx^{i,l}_t = -\lambda (x^{i,l}_t - {\bar x}^{*,l}_t)dt  + \sigma   (x^{i,l}_t - {\bar x}_t^{*,l}) dW_t^l.
\end{aligned}
\end{align}
Now, we use the following quadratic variation relations:
\[ dt \cdot dt = 0, \quad dt \cdot dW^l_t = 0, \quad  dW^l_t  \cdot dt = 0, \quad  dW_t^l \cdot dW_t^k = \delta_{lk} dt \]
and \eqref{D-1} to see
\[
dx^k_t \cdot  dx^l_t = \sigma^2   \delta_{kl}   (x^{i,k}_t - {\bar x}_t^{*,k})(x^{i,l}_t - {\bar x}_t^{*,l}) dt.
\]
This certainly implies the desired estimate.
\end{proof}

\vspace{0.2cm}

In what follows, we use a handy notation for partial derivatives:
\[ \partial_l  := \frac{\partial}{\partial x^l}, \quad \partial^2_{kl} := \frac{\partial^2}{\partial x^k \partial x^l}, \quad  l, k = 1, \cdots, d. \]
Let $L = L(x)$ be a $C_b^2$-objective function satisfying the following relations:
\begin{equation} \label{D-7}
 L_m:=\inf_{x \in \bbr^d} L(x) >0 \quad \mbox{and} \quad C_L:=\max\left\{\sup_{x\in\bbr^d}\|\nabla^2L(x)\|_2,\max_{1\leq l\leq d}\sup\limits_{x\in\bbr^d}|\partial^2_{l}L(x)|\right\} <\infty,
\end{equation} 
where $\|\cdot\|_2$ denotes the spectral norm. 
First, note that
\begin{align}
\begin{aligned} \label{D-7-1}
& \partial_k \Big( e^{-\beta L(X_t^i)} \Big) = -\beta e^{-\beta L(X_t^i)} \partial_k L(X_t^i), \\
& \partial_{kl}^2 \Big( e^{-\beta L(X_t^i)} \Big) = \beta e^{-\beta L(X_t^i)} \Big[  \beta \partial_l L(X_t^i)  \cdot \partial_k L(X_t^i) - \partial_{lk}^2 L(X_t^i) \Big ].
\end{aligned}
\end{align}
Now, we apply Ito's formula to $\frac{1}{N}\sum_{i=1}^N e^{-\beta L(X_t^i)}$ using the relations \eqref{D-7-1} to get 
\begin{align}
\begin{aligned} \label{D-8}
&d\left( \frac{1}{N}\sum_{i=1}^N e^{-\beta L(X_t^i)}\right) \\
&  = \frac{1}{N}\sum_{i=1}^N \Big[ \sum_{k=1}^d \partial_k \Big( e^{-\beta L(X_t^i)} \Big) dx_t^{i,k} + \frac{1}{2} \sum_{k,l=1}^d  \partial_{kl}^2\Big( e^{-\beta L(X_t^i) } \Big) dx_t^{i,k} \cdot dx_t^{i,l}                     \Big ] \\
&  = \frac{1}{N}\sum_{i=1}^N \Big[ -\beta e^{-\beta L(X_t^i)}\sum_{k=1}^d \partial_k L(X_t^i) \Big( -\lambda (x^{i,k}_t - {\bar x}^{*,k}_t)dt  + \sigma    (x^{i,k}_t - {\bar x}_t^{*,k}) dW_t^k\Big)\Big]\\
& \hspace{0.5cm} + \frac{1}{N}\sum_{i=1}^N \frac{1}{2}\beta e^{-\beta L(X_t^i)} \Big[\sum_{k,l=1}^d  \Big(  \beta \partial_l L(X_t^i)  \cdot 
\partial_k L(X_t^i) - \partial_{lk}^2 L(X_t^i)  \Big ) \sigma^2   \delta_{kl}   (x^{i,k}_t - {\bar x}_t^{*,k})(x^{i,l}_t - {\bar x}_t^{*,l}) dt  \Big ] \\
&  =\frac{1}{N}\sum_{i=1}^N \beta e^{-\beta L(X_t^i)}\nabla L(X_t^i)\cdot\Big[ \lambda (X_t^i-\bar X_t^*)dt-\sigma \sum_{k=1}^d(x^{i,k}_t - {\bar x}_t^{*,k}) dW_t^k e_k \Big] \\
& \hspace{0.5cm} +\frac{1}{N}\sum_{i=1}^N\Big[ \beta e^{-\beta L(X_t^i)}\sum_{k=1}^d \big(-\partial_{kk}L(X_t^i)+\beta(\partial_k L(X_t^i))^2\big)\frac{1}{2}\sigma^2(x_t^{i,k}-\bar x_t^{*,k})^2 \Big] dt.
\end{aligned}
\end{align}
We take expectations on both sides of \eqref{D-8} to get
\begin{align}\label{D-9}
\begin{aligned}
&d\left(\frac{1}{N} \sum_{i=1}^N \mathbb E e^{-\beta L(X_t^i)}\right) \\
&  =\frac{1}{N}\sum_{i=1}^N \mathbb E\Big[  \beta e^{-\beta L(X_t^i)}\nabla L(X_t^i)\cdot \lambda (X_t^i-\bar X_t^*) \Big] dt \\
& \hspace{0.5cm} +\frac{1}{N}\sum_{i=1}^N\mathbb E\Big[ \beta e^{-\beta L(X_t^i)}\sum_{k=1}^d \Big(-\partial_{kk}L(X_t^i)+\beta(\partial_k L(X_t^i))^2\Big)\frac{1}{2}\sigma^2(x_t^{i,k}-\bar x_t^{*,k})^2 \Big] dt\\
&  =:{\mathcal I}_{21} dt+ {\mathcal I}_{22} dt.
\end{aligned}
\end{align}
Below, we estimate the terms ${\mathcal I}_{2i},~i=1,2$ as follows.\newline

\begin{lemma} \label{L4.3}
Let $\{X_t^i \}$ be a solution to \eqref{D-1}. Then, the term ${\mathcal I}_{2i},~i=1,2$ satisfies
\begin{eqnarray*}
&& (i)~{\mathcal I}_{21} \geq  -\lambda C_L\beta e^{-\beta L_m}\frac{1}{N}\sum_{i=1}^N \mathbb E |X_t^i-\bar X_t^*|^2; \cr
&& (ii)~{\mathcal I}_{22} \geq -\frac{1}{2}\sigma^2 C_L\beta e^{-\beta L_m}\frac{1}{N}\sum_{i=1}^N\mathbb E|X_t^i-\bar X_t^*|^2.
\end{eqnarray*}
\end{lemma}
\begin{proof} Below, we estimate ${\mathcal I}_{2i}$ separately. \newline

\noindent $\bullet$~(Estimate of ${\mathcal I}_{21}$): First, we use definition of $\bar X_t^*$ to see 
\[ \bigg(\sum_{i=1}^N e^{-\beta L(X_t^i)}\bigg)\bar X_t^*=\sum_{i=1}^N e^{-\beta L(X_t^i)}X_t^i. \]
This yields
\begin{equation} \label{D-9-1}
  \sum_{i=1}^N e^{-\beta L(X_t^i)}\nabla L(\bar X_t^*)\cdot (\bar X_t^*-X_t^i)=0.
 \end{equation} 
Then, we use \eqref{D-7} and \eqref{D-9-1} to find
\begin{align}
\begin{aligned} \label{D-10}
{\mathcal I}_{11} & =\frac{\beta \lambda}{N}\sum_{i=1}^N \mathbb E\Big[ e^{-\beta L(X_t^i)}\nabla L(X_t^i)\cdot (X_t^i-\bar X_t^*) \Big] dt \\
&=\frac{\beta \lambda}{N}\sum_{i=1}^N \mathbb E\Big[ e^{-\beta L(X_t^i)}\big(\nabla L(X_t^i)- \nabla L(\bar X_t^*)\big)\cdot (X_t^i-\bar X_t^*) \Big] \\
&\geq -\lambda C_L\beta e^{-\beta L_m}\frac{1}{N}\sum_{i=1}^N \mathbb E |X_t^i-\bar X_t^*|^2.
\end{aligned}
\end{align}

\noindent $\bullet$~(Estimate of ${\mathcal I}_{22}$): By direct calculation, one has
\begin{align}
\begin{aligned} \label{D-11}
{\mathcal I}_{12} &=-\frac{\sigma^2 \beta}{2N}\sum_{i=1}^N\mathbb E\Big[ e^{-\beta L(X_t^i)}\sum_{k=1}^d \partial_{kk}L(X_t^i) (x_t^{i,k}-\bar x_t^{*,k})^2 \Big] \\
&\geq -\frac{1}{2}\sigma^2 C_L\beta e^{-\beta L_m}\frac{1}{N}\sum_{i=1}^N\mathbb E|X_t^i-\bar X_t^*|^2 .
\end{aligned}
\end{align}
\end{proof}
Now, we are ready to provide the convergence result of the continuous CBO algorithm. Note that in \cite{C-J-L-Z}, the $L^2(\Omega)$-limit of the stochastic process $X_t$ was actually equal to some non-random $\tilde x\in\bbr^d$, but it is not the case for our $N$-particle model. This resulted in the statement of Theorem \ref{T4.1} slightly different from the analogous theorem (Theorem 3.1) in \cite{C-J-L-Z}.
\begin{theorem} \label{T4.1}
Suppose that $\lambda, \sigma$ and $\{X_0^i \}$ satisfy
\begin{align*}
\begin{aligned}
& 2\lambda>\sigma^2, \quad X_0^i: i,i.d,, \quad X_0^i \sim X^{in}~~\mbox{for some random variable $X^{in}$}, \cr
&  (1-\varepsilon)\mathbb E \Big[ e^{-\beta L(X^{in})} \Big ] \geq \frac{2\lambda+\sigma^2}{2\lambda-\sigma^2} C_L\beta e^{-\beta L_m}\sum_{l=1}^d \mathbb E \Big[ \max_{1\leq i\leq N} (x_0^{i,l} -{\bar x}^l_0)^2 \Big],
\end{aligned}
\end{align*}
for some $0<\varepsilon<1$. Then, one has 
\[
\operatorname{ess~inf}_{\omega\in\Omega} L(X^\infty(\omega))\leq \operatorname{ess~inf}_{\omega\in\Omega} L(X^{in}(\omega))+ {\mathcal O}\Big(\frac{1}{\beta}\Big),\quad \mbox{for $\beta \gg 1$}.
\]
Consequently, if the global minimizer $X^*$ of $L$ is contained in $\operatorname{supp}\operatorname{law}(X^{in})$, then
\[
\operatorname{ess~inf}_{\omega\in\Omega} L(X^\infty(\omega))\leq L_m+ {\mathcal O}\Big(\frac{1}{\beta}\Big),\quad \mbox{for $\beta \gg 1$}.
\]
\end{theorem}
\begin{proof}

In \eqref{D-9}, we use \eqref{D-10}, \eqref{D-11} and Lemma \ref{L4.1} (iii) to find 
\begin{align}\label{D-12}
\begin{aligned}
\frac{d}{dt}\left( \frac{1}{N}\sum_{i=1}^N \mathbb E e^{-\beta L(X_t^i)}\right)&\geq-\left(\lambda+\frac{1}{2}\sigma^2\right) C_L\beta e^{-\beta L_m}\frac{1}{N}\sum_{i=1}^N\mathbb E|X_t^i-\bar X_t^*|^2\\
&\geq -2\left(\lambda+\frac{1}{2}\sigma^2\right) C_L\beta e^{-\beta L_m}e^{-(2\lambda-\sigma^2)t} \sum_{l=1}^d \mathbb E \Big[ \max_{1\leq i\leq N} (x_0^{i,l} -{\bar x}^l_0)^2 \Big].
\end{aligned}
\end{align}
Now integrating \eqref{D-12} in $t$ gives 
\begin{align}
\begin{aligned} \label{D-13}
&\frac{1}{N}\sum_{i=1}^N \mathbb E e^{-\beta L(X_t^i)} \\ 
&\geq \frac{1}{N}\sum_{i=1}^N \mathbb E e^{-\beta L(X_0^i)} -2\left(\lambda+\frac{1}{2}\sigma^2\right) C_L\beta e^{-\beta L_m}\frac{1-e^{-(2\lambda-\sigma^2)t}}{2\lambda-\sigma^2}\sum_{l=1}^d  \mathbb E \Big[ \max_{1\leq i\leq N} (x_0^{i,l} -{\bar x}^l_0)^2 \Big] \\
&\geq \frac{1}{N}\sum_{i=1}^N \mathbb E e^{-\beta L(X_0^i)} -\frac{2\lambda+\sigma^2}{2\lambda-\sigma^2} C_L\beta e^{-\beta L_m}\sum_{l=1}^d \mathbb E \Big[ \max_{1\leq i\leq N} (x_0^{i,l} -{\bar x}^l_0)^2 \Big].
\end{aligned}
\end{align}
Letting $t\to\infty$, and we use Lemma \ref{L4.2} (i) to find
\[
\mathbb Ee^{-\beta L(X_\infty)} \geq \frac{1}{N}\sum_{i=1}^N \mathbb E e^{-\beta L(X_0^i)} -\frac{2\lambda+\sigma^2}{2\lambda-\sigma^2} C_L\beta e^{-\beta L_m}\sum_{l=1}^d \mathbb E \Big[ \max_{1\leq i\leq N} (x_0^{i,l} -{\bar x}^l_0)^2 \Big] \geq \varepsilon\mathbb E e^{-\beta L(X^{in})},
\]
i.e.,
\[
-\frac{1}{\beta}\log \mathbb E e^{-\beta L(X^{\infty})}\leq-\frac{1}{\beta}\log \mathbb E e^{-\beta L(X^{in})}-\frac{1}{\beta}\log\varepsilon.
\]
Now  Laplace's principle implies
\[
\operatorname{ess~inf}_{\omega\in\Omega} L(X^\infty(\omega))\leq \operatorname{ess~inf}_{\omega\in\Omega} L(X^{in}(\omega))+O\Big(\frac{1}{\beta}\Big) \quad \mbox{for $\beta \gg 1$}.
\]
\end{proof}

\section{Numerical simulations} \label{sec:5}
\setcounter{equation}{0}
In this section, we conduct several numerical tests to verify the results of the convergence analysis. For a numerical test, we use the Rastrigin function as in \cite{C-J-L-Z,P-T-T-M} as the objective function:
\[ L(X) =  \sum_{i=1}^{d} \Big[  (x^i - B)^2 - 10 \cos(2\pi(x^i - B)) + 10    \Big]  + C,        \]
where constants $B$ and $C$ are given by
\[ B:= \mbox{argmin}~L(X), \quad C:= \mbox{min}~L(X). \]
Note that this function has a unique global minimizer, namely $X=(B, \cdots, B) \in \bbr^d$. However, it has many local minimizers as can be seen Figure 1 (see the graph for $L$ is provided in Figure \ref{Fig1} with $d = 2, B = C= 0$).  \newline
\begin{figure}[h!]
\centering
\includegraphics[width=1\textwidth]{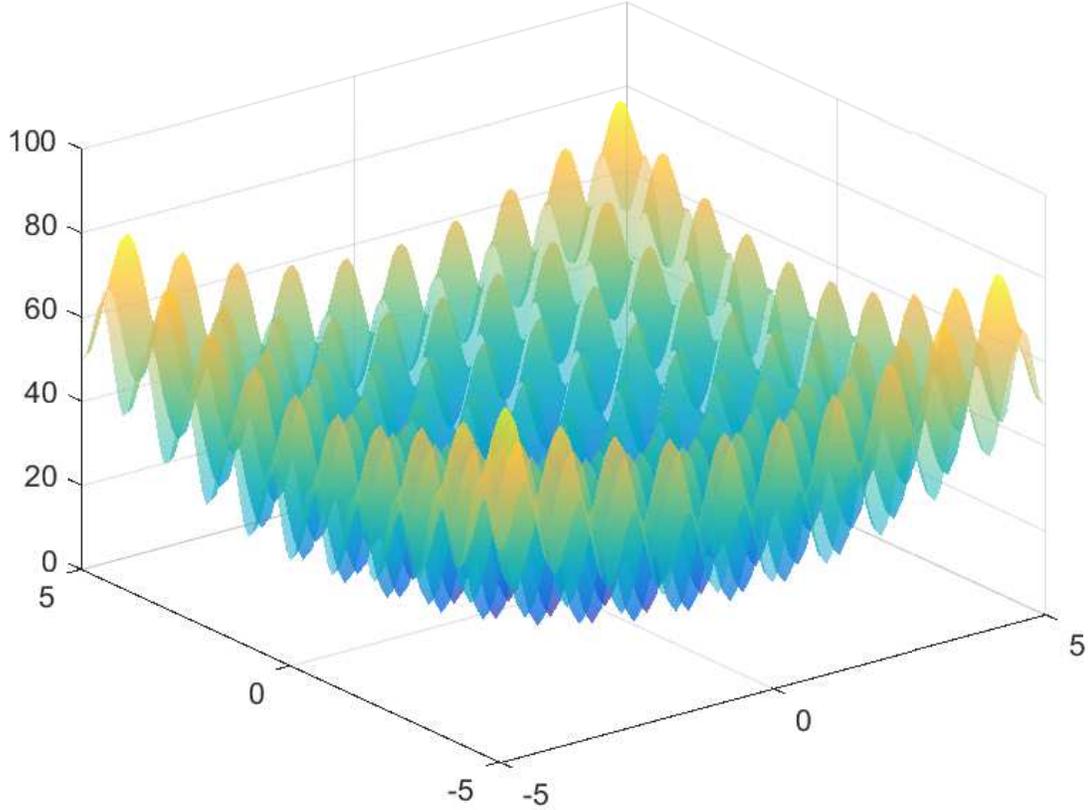}
\caption{A 3D plot of the Rastrigin function}
\label{Fig1}
\end{figure}
For the initial data and system parameters, we choose $N=100$ points uniformly from the square $[-2,2]\times[-2,2]$, which includes a global minimum point, and use parameters:
\[
\Delta t= h = 0.01,\quad N=100, \quad \beta=10,\quad \lambda=1.
\]
For the same chosen initial data set as above, we perform simulations for $\sigma=0,1,2$ and compare the results.
\subsection{Continuous algorithm} For the simulations of the continuous algorithm, we use the following two-step numerical scheme:
\begin{align*}
\begin{aligned}
&\hat X_n^i=\bar X_n^*+(X_n^i-\bar X_n^*)e^{-\lambda h},\\
&X_{n+1}^i=\hat X_n^i+\sigma\sqrt{h}\sum_{l=1}^d (\hat x_n^{i,l}-\bar x_n^{*,l}) w_n^l e_l, 
\end{aligned}
\end{align*}
where $w_n^l$ $(l=1,\cdots,d, n=0,1,2,\cdots)$ are independent and follow the standard normal distribution, and $\Delta t$ is the time step. In Section \ref{sec:3}, we derived the following explicit formula for the continuous algorithm:
\[
x_t^{i,l} -x_t^{j,l} =(x_0^{i,l}-x_0^{j,l})\exp\left[ -\Big(\lambda+\frac{1}{2}\sigma^2\Big)t+\sigma W_t^l\right ].
\]
It is easy to see  from the above formula that the particles will reach a global consensus almost surely if and only if the coupling strength and noise intensity satisfy
\[ \lambda>0\quad\mbox{or}\quad |\sigma|>0. \]

Note that for larger $\sigma$, the speed of consensus on average is faster. Note that these facts do not require the condition $2\lambda>\sigma^2$ as can be seen in convergence analysis in Section \ref{sec:4}. 
One can observe this result numerically. In Figures \ref{Fig2},\ref{Fig3} and \ref{Fig4}, we plot the positions of the particles for $\sigma=0,1,2$, respectively. Indeed, the particles seem to converge faster, as $\sigma$ increases. In Figure \ref{Fig5}, we plot a sample path of $\log|x_t^{1,1}-x_t^{2,1}|$ for $\sigma=0,1,2$. As expected, the graph for $\sigma=0$ is linear, and the function eventually decays faster for  large $\sigma$.
\begin{figure}[h!]
\centering
\includegraphics[width=1\textwidth]{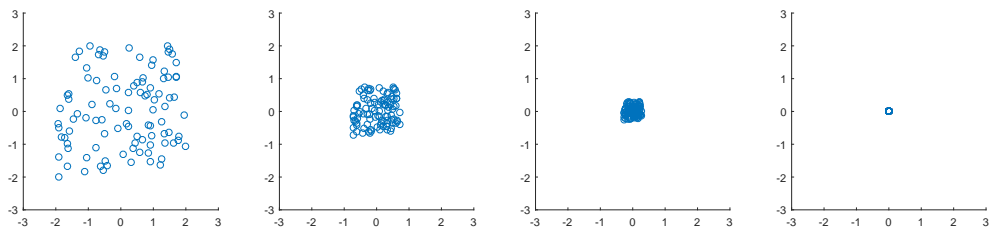}
\caption{Temporal evolution of state configuration for $t=0,1,2,10$ $(\sigma=0)$.}
\label{Fig2}
\end{figure}

\begin{figure}[h!]
\centering
\includegraphics[width=1\textwidth]{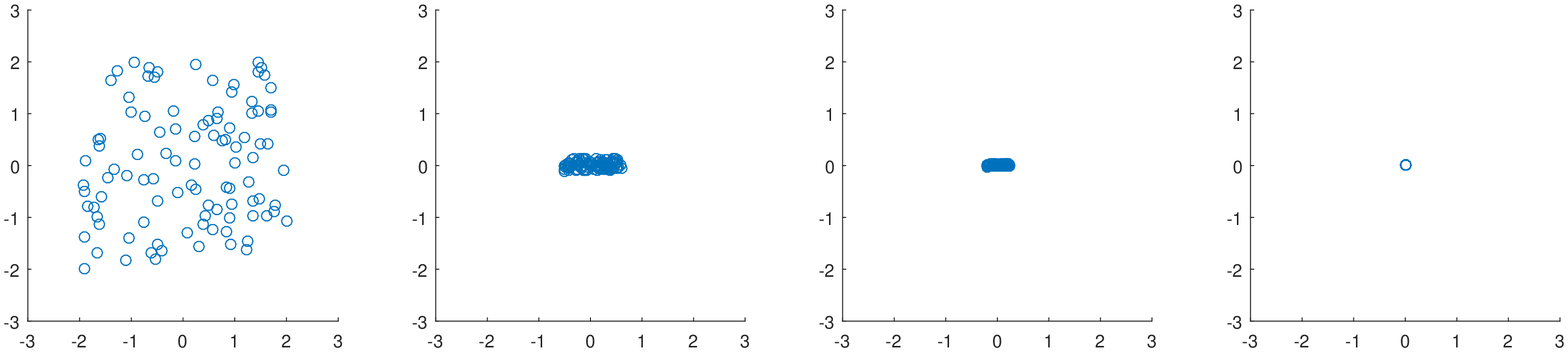}
\caption{Temporal evolution of state configuration for $t=0,1,2,10$ $(\sigma=1)$.}
\label{Fig3}
\end{figure}

\begin{figure}[h!]
\centering
\includegraphics[width=1\textwidth]{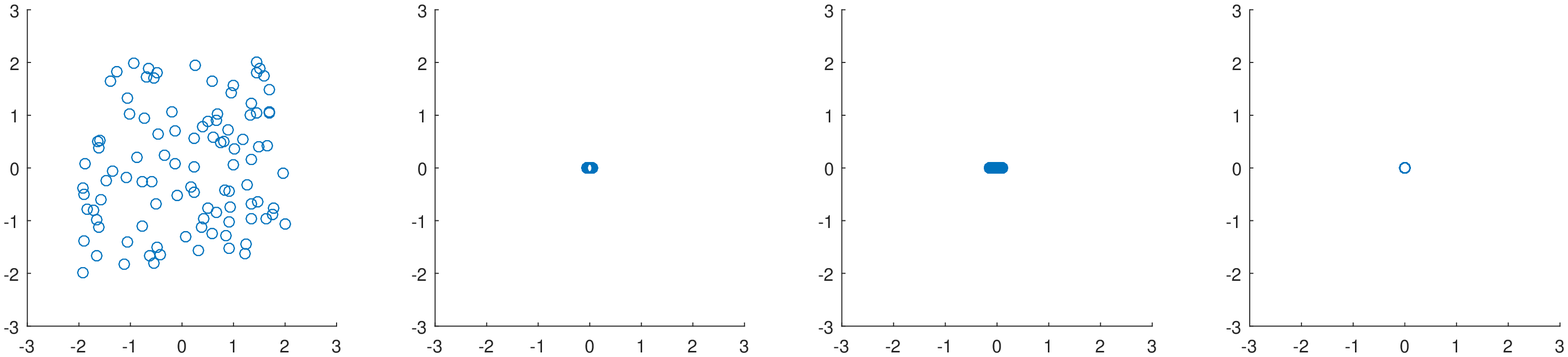}
\caption{Temporal evolution of state configuration for $t=0,1,2,10$ $(\sigma=2)$.}
\label{Fig4}
\end{figure}

\begin{figure}[h!]
\centering
\includegraphics[width=1\textwidth]{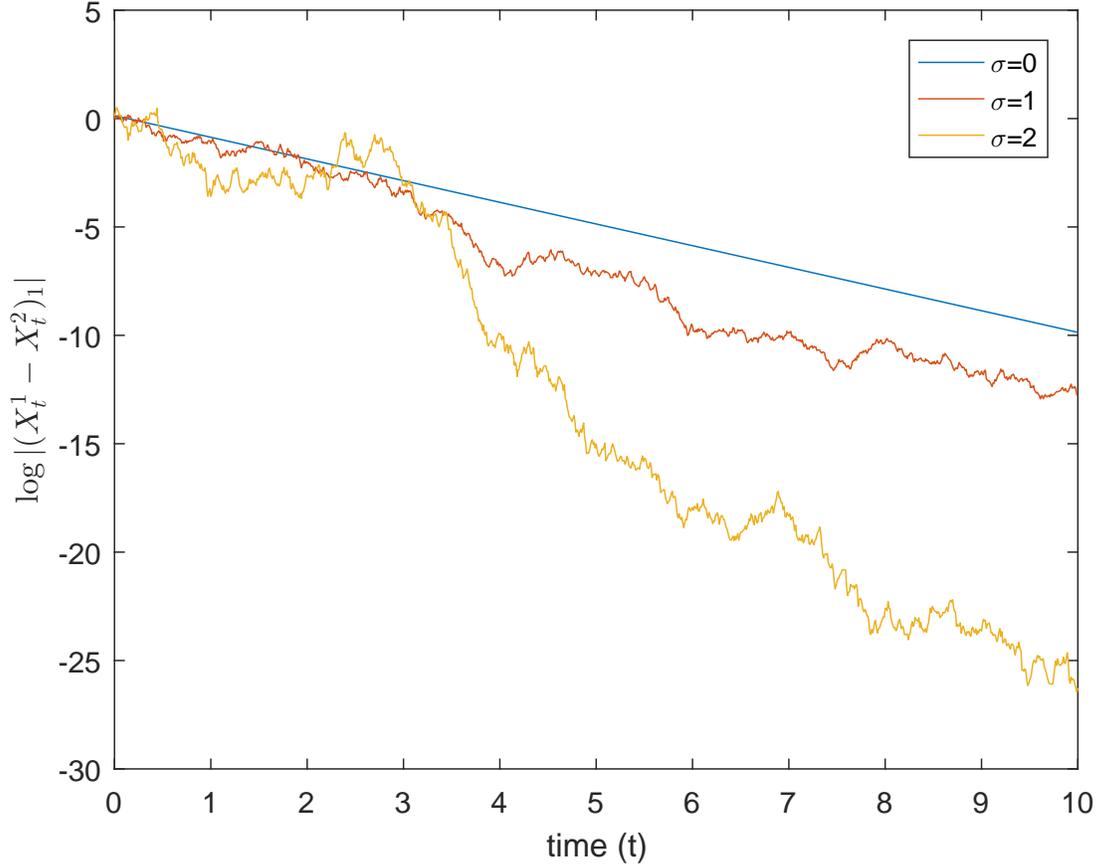}
\caption{Graph of $\log|x_t^{1,1}-x_t^{2,1}|$ for $\sigma=0,1,2$.}
\label{Fig5}
\end{figure}

\subsection{Discrete algorithm}
Recall the discrete algorithm:
\begin{equation} \label{F-1}
\begin{cases}
\displaystyle X^i_{n +1} = X^i_n  + \lambda h \sum_{k = 1}^{N} \psi_n^k (X_n^k - X_n^i) + \sigma \sqrt{h} \sum_{k = 1}^{N} \sum_{l=1}^{d} \psi_n^k (x_n^{k,l} - x_n^{i,l})  Z^l_n e_l, \\
\displaystyle \psi^k_n := \frac{e^{-\beta L(X^k_n)}}{\sum_{i=1}^{N} e^{-\beta L(X^i_n)}}, \quad 1 \leq k \leq N, \quad  n = 0, 1, \cdots.
\end{cases}
\end{equation}
In this subsection, we study the formation of global consensus for the discrete algorithm \eqref{F-1} numerically. In Theorem \ref{T3.4}, we have shown that if $m:=\lambda-\frac{\sigma^2}{2}-\frac{\lambda^2h}{2}>0$ then the quantity $\Delta_n^{ij} := |X_n^i - X_n^j|$ tends to zero almost surely, as $n \to \infty$ with a decay rate approximately $\exp(-\frac{1}{2}mnh)$. Figure \ref{Fig6} indicates that this result is not optimal. To compute $\mathbb E\Big[ \log|x_t^{1,1}-x_t^{2,1}| \Big] $, we simulated 100 sample paths and then took average of those paths. Although $2 \lambda <\sigma^2$ for $\sigma=2$, $\Delta_n^{ij}$ converges in this case. Moreover, although the decay rate obtained in Theorem \ref{T3.4} decreases, as $\sigma$ increases, one can see that the decay rate increases as $\sigma$ increases.

\begin{figure}[h!]
\centering
\includegraphics[width=1\textwidth]{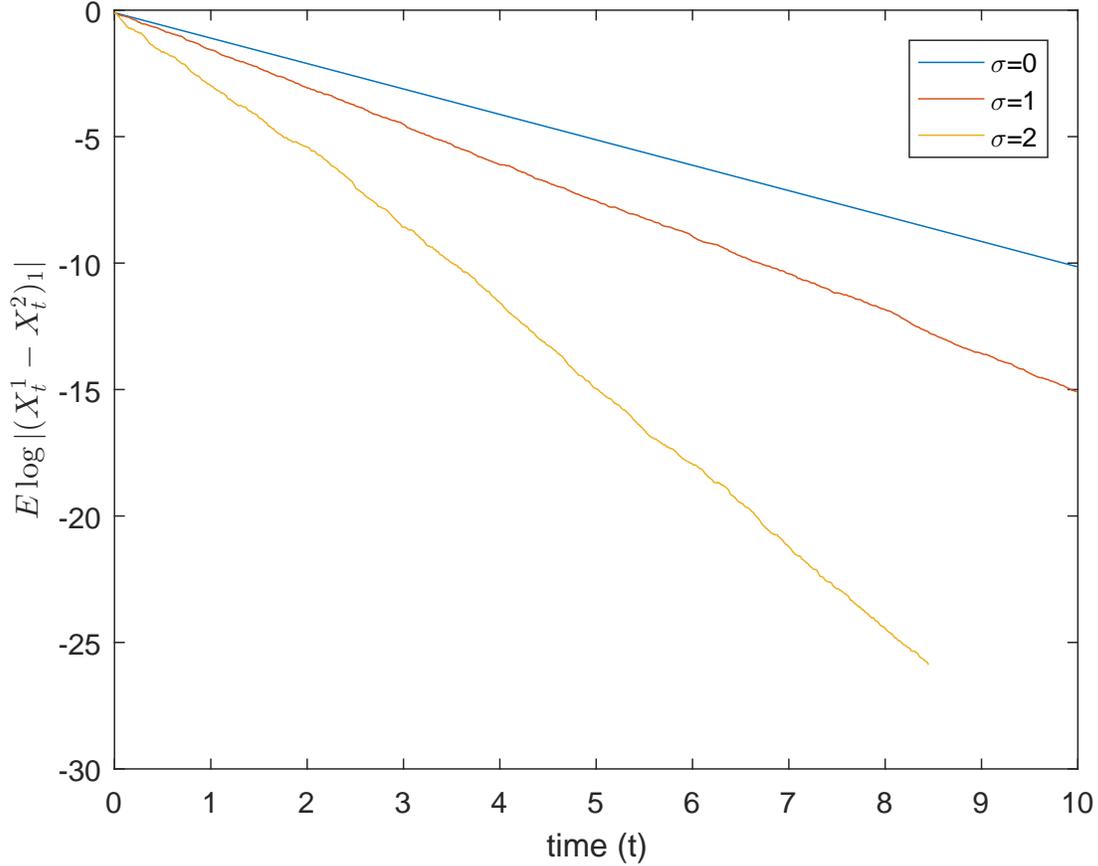}
\caption{Graph of $\mathbb E \Big[ \log|x_n^{1,1}-x_n^{2,1}| \Big]$ for $\sigma=0,1,2$.}
\label{Fig6}
\end{figure}

\section{conclusion} \label{sec:6}
\setcounter{equation}{0}
In this paper, we have provided a rigorous convergence analysis for the first-order consensus-based optimization algorithm. In \cite{C-J-L-Z}, the convergence was understood using the corresponding mean-field limit, the Fokker-Planck equation. Thus, the convergence  of the original CBO algorithm remains unresolved there. The main contribution of this work is to provide the convergence analysis directly on the  CBO algorithm model
at the  particle level. After rewriting the given continuous optimization algorithm into a first-order consensus form, we use the detailed structure of the coupling term to derive an exact formula for the state differences. Our explicit formula shows that global consensus will emerge for any initial data, whereas in order to prove the convergence toward a global minimum,  we need a sufficient condition--which is dimension independent-- for systems parameters and initial data to show that the global consensus state tends to a global minimum, as the reciprocal of temperature tends to infinity using Laplace's principle. The emergence of global consensus will emerge for continuous and discrete algorithms. However, we can obtain the convergence analysis only for the continuous algorithm due to Ito's stochastic analysis for twice differentiable and bounded objective functions. In contrast, for discrete algorithm, we do not have available mathematical tools to derive convergence analysis at present. 

There are several issues which we cannot deal with in this paper. To name a few, it will be interesting to relax the regularity of the objective function to the less regular objective function, at least continuous one. Finally, random batch methods were used to reduce the computational cost of the $N$-term summation in \cite{C-J-L-Z, J-L-L} in which convergence remains as an open question.

Moreover, it will be interesting to see whether our presented analysis can be applied to other metaheuristic algorithms based on the swarm intelligence. These issues will be discussed in a future work.

\end{document}